\documentclass[12pt]{amsart}

\pdfoutput=1

\usepackage{enumerate}
\usepackage[nocompress]{cite}
\usepackage{hyperref}
\usepackage[margin=3.14cm]{geometry}
\usepackage{microtype}

\theoremstyle{plain}
\newtheorem{theorem}{Theorem}[section]
\newtheorem{proposition}[theorem]{Proposition}
\newtheorem{lemma}[theorem]{Lemma}
\newtheorem{corollary}[theorem]{Corollary}
\newtheorem{conjecture}[theorem]{Conjecture}

\theoremstyle{definition}
\newtheorem{definition}[theorem]{Definition}
\newtheorem{remark}[theorem]{Remark}
\newtheorem{example}[theorem]{Example}

\newcommand		{\p}[1]	{\left(#1\right)}
\newcommand		{\ppp}[1] {\left\{#1\right\}}
\newcommand		{\abs}[1]{\left|#1\right|}
\newcommand 	{\norm}[1]{\left\lVert#1\right\rVert}
\newcommand    	{\R} {\mathbb R}
\newcommand    	{\C} {\mathbb C}
\newcommand    	{\N} {\mathbb N}
\newcommand    	{\Z} {\mathbb Z}
\newcommand 	{\bounded} {\mathbb B}
\renewcommand   {\O} {\mathcal O}
\newcommand    	{\as} {\quad \text{as} \ }
\DeclareMathOperator*{\argmax}{arg\,max}

\title{Bounded Power Series on the Real Line}
\author{Davide Sclosa}
\address{Vrije Universiteit Amsterdam, De Boelelaan 1111, 1081 HV Amsterdam, The Netherlands}
\date{\today}

\begin{document}

\begin{abstract}
We investigate power series that converge to a bounded function on the real line.
First, we establish relations between coefficients of a power series
and boundedness of the resulting function;
in particular, we show that boundedness can be prevented by certain Tur\'an inequalities
and, in the case of real coefficients, by certain sign patterns.
Second, we show that the set of bounded power series naturally supports three
topologies and that these topologies are inequivalent and incomplete.
In each case, we determine the topological completion.
Third, we study the algebra of bounded power series, revealing the
key role of the backward shift operator.
\end{abstract}

\maketitle
\thispagestyle{empty}


\section{Introduction}
Relating the coefficients of a power series to the values of the corresponding
function has a long history~\cite{szego1936some}.
In 1885, Cauchy proved that
a power series bounded on the complex plane must be constant,
and thus all but one coefficients must be equal to zero~\cite{cauchy1885algebraische};
the result is now known as Liouville's theorem.
In 1917, Schur characterized the coefficients of power series bounded in the
interior of the unit circle~\cite{dym2003contributions, schur1917potenzreihen, szego1936some}.
In~1940, Boas studied power series of exponential type bounded on the real
line~\cite[Ch.\ 7,8]{boas2011entire}.
Despite the popularity of some power series bounded on the real line
but not of exponential type,
like the Gaussian~$\exp(-\pi x^2)$ or the Uppuri-Carpenter generating
function~$\exp(1-\exp(x))$~\cite{uppuluri1969numbers,de2016wilf,amdeberhan2013complementary},
this general case remains largely unexplored.

We say that a power series
\begin{equation} \label{eq:main}
	f(x) = a_0 + a_1 x + a_2 x^2 + \cdots + a_n x^n + \cdots
\end{equation}
is \emph{bounded on the real line} if it converges for every~$x\in \R$
and if the corresponding function~$f:\R\to\C$ is bounded.
Boundedness on a line is a remarkable property: infinitely many monomials
have to balance each other out, uniformly, on an unbounded set.
In contrast to convergence, boundedness is not a tail property:
it is not preserved by changing finitely many coefficients.
Indeed, dividing~\eqref{eq:main} by~$x^n$ for some~$n\geq 1$ gives
\begin{equation} \label{eq:tails_determine_heads}
	a_n = -\lim_{\abs{x}\to \infty} \p{\sum_{k\geq 1} a_{n+k} x^k}, \quad \forall n\geq 1,
\end{equation}
implying that the tail~$(a_k)_{k>n}$ determines~$a_1,\ldots,a_n$
(the coefficient~$a_0$ is arbitrary).

On the opposite, tails are not uniquely determined by heads.
Indeed, for every polynomial~$p(x) = a_0 + a_1 x + \cdots + a_n x^n$ the power series
\begin{equation} \label{eq:everything_head}
	p(x) \exp(-x^{2n})
		= a_0 + a_1 x + \cdots + a_n x^n + \cdots
\end{equation}
is bounded on the real line, implying that every finite complex sequence is the head of a
bounded power series, and thus that tails are not uniquely determined.

While boundedness is not a tail property itself, it entails some tail
properties. In Section~\ref{sec:convexity} we related boundedness
to certain \emph{Tur\'an inequalities}~\cite{szego1948inequality, karp2010log}, namely
\begin{equation} \label{eq:convexity}
	\abs{a_{n+1} a_{n-1}} \leq \chi \abs{a_{n}}^2, \quad a_{n-1}\neq 0, \quad \forall n\geq N,
\end{equation}
where~$N\geq 1$ and~$\chi>0$ are constants. The constant~$\chi$ controls the logarithmic
concavity of the sequence~$(a_n)_{n\geq 0}$.
For simplicity,
we call a power series \emph{real} (\emph{complex}) if its coefficients are real (complex).
This is the main result of Section~\ref{sec:convexity}:

\begin{theorem} \label{thm:convexity}
Let~$f(x)=\sum_{n\geq 0} a_n x^n$ be a real power series, bounded on the real line,
and satisfying~\eqref{eq:convexity} for
some~$N\geq 1$ and~$0<\chi< 1$. Let
\[
	\vartheta(\chi) = \sup \ppp{L \in \Z : \sum_{1\leq k\leq L} \chi^{\frac{k^2}{2}} < 1},
	\quad
	\psi(\chi) = \inf \ppp{L\in \Z : \sum_{k> L} \chi^{\frac{k^2}{2}} < \frac{1}{2}},
\]
and let~$L(n)$ denote the supremum of the integers~$L\geq 0$
for which the tail~$(a_k)_{k\geq n}$
contains~$L$ consecutive elements with the same sign. Then
\begin{align} \label{eq:max_min}
	\vartheta(\chi) < \lim_{n\to\infty} L(n) \leq 2 \psi(\chi).
\end{align}
\end{theorem}

The proof of Theorem~\ref{thm:convexity} relies on a discrete Legendre transform,
see Remark~\ref{rem:Legendre}.
This theorem is useful to show that certain power series
are not bounded:

\begin{corollary} \label{cor:convexity}
A real power series satisfying~\eqref{eq:convexity}
for some~$N\geq 1$ and~$0<\chi\leq \frac{1}{2}$
is unbounded on the real line.
\end{corollary}
\begin{proof}
Let~$0<\chi\leq \frac{1}{2}$.
By~\eqref{eq:max_min},
it is enough to prove that~$\vartheta(\chi)\geq 2$ and~$\psi(\chi)\leq 1$.
The inequality~$\vartheta(\chi)\geq 2$ follows
from~$\chi^{\frac{1}{2}}+\chi^{\frac{2^2}{2}}\leq 2^{-\frac{1}{2}} + 2^{-2} \leq 1$.
To prove~$\psi(\theta)\leq 1$, let~$S= \sum_{k>1} \chi^{\frac{k^2}{2}}$ and note that
\[
	S < \chi^{\frac{2^2}{2}} \p{1 + \chi + \sum_{k>3} \chi^{\frac{k^2-2^2}{2}}}
	< \chi^{\frac{2^2}{2}} \p{1 + \chi + S} \leq \frac{1}{4} \p{1+\frac{1}{2} + S}.
\]
It follows that~$S<\frac{1}{2}$ and therefore that~$\psi(\theta)\leq 1$, completing the proof.
\end{proof}

The constant~$1/2$ in Corollary~\ref{cor:convexity}, which bounds the coefficient concavity
of bounded power series, is not sharp: determining the optimal bound
is proposed as an open problem at the end of Section~\ref{sec:convexity}.

Corollary~\ref{cor:convexity} implies that~$\sum_{n\geq 0} \pm \rho^{n^2} x^n$ is unbounded
if~$0\leq \rho\leq 2^{-\frac{1}{2}}$ regardless of the choice of signs;
this generalizes the case~$\rho=1/3$ proven by P.~Majer~\cite{mathoverflow}.

The coefficient sequence of a real bounded power series
must change sign infinitely many times. If Tur\'an inequalities hold, something stronger
is true:

\begin{corollary} \label{cor:signs}
Let~$f(x)=\sum_{n\geq 0} a_n x^n$ be a real power series, bounded on the real line,
and satisfying~\eqref{eq:convexity} for
some constants~$N\geq 1$ and~$0<\chi< 1$.
Then~$(a_n)_{n\geq 0}$ does not contain
arbitrarily long subsequences of consecutive coefficients with the same sign,
nor arbitrarily long subsequences of consecutive coefficients alternating sign.
\end{corollary}
\begin{proof}
The first statement follows from~\eqref{eq:max_min} since~$L(0)\leq 2\psi(\chi)<\infty$.
The second statement follows from replacing~$f(x)$ by~$f(-x)$.
\end{proof}

The techniques developed in Section~\ref{sec:convexity}
might help deducing inequalities
for combinatorial sequences from the boundedness of their generating function.
In contrast to Corollary~\ref{cor:signs},
we conjecture that~$\exp(1-\exp(x))$
contains arbitrarily long subsequences of consecutive coefficients with the same sign;
see the related Wilf's Conjecture~\cite{de2016wilf, amdeberhan2013complementary} and
Remark~\ref{rem:combinatorics} for details.

Let~$\bounded$ denote the set of complex power series bounded on the real line.
Equation~\eqref{eq:everything_head} suggests
that the set~$\bounded$ is dense in the set of complex power series with respect to
some topology.
Section~\ref{sec:topology} discusses three natural
topologies on~$\bounded$:

\begin{theorem} \label{thm:topology}
Given a power series bounded on the real line~$f(x)=\sum_{n\geq 0} a_n x^n$, let
\[
	\norm{f}_{\ell^1} = \sum_{n\geq 0} \abs{a_n},
	\quad \norm{f}_{\infty,\R} = \sup_{x\in \R} \abs{f(x)}.
\]
The following facts are true:
\begin{enumerate} [(i)]
\item The set~$\bounded$ endowed with~$\norm{\cdot}_{\ell^1}$ is (isometric to)
a dense subset of~$\ell^1(\N,\C)$; \label{thm:topology:l1}
\item The set~$\bounded$ endowed with~$\norm{\cdot}_{\infty,\R}$ is (isometric to)
a dense subset of~$C^\mathrm{b}(\R,\C)$, the set of bounded continuous complex-valued
functions on the real line;
\label{thm:topology:R}
\item The set~$\bounded$ endowed with uniform convergence on compact subsets of~$\C$
is (isometric to) a dense subset of the set entire functions. \label{thm:topology:C}
\end{enumerate}
Moreover, the three topologies are inequivalent.
\end{theorem}

Section~\ref{sec:algebra} concerns the algebraic structure of~$\bounded$.
Taking~$n=1$ in~\eqref{eq:tails_determine_heads} shows that~$\bounded$ is invariant
under backward shift~$\sigma\p{(a_n)_{n\geq 0}} = (a_{n+1})_{n\geq 0}$.
Note that the backward shift plays a central role in the theory of power series
in the unit circle~\cite{cima2000backward, nikolski2019hardy}.
The following theorem collects the main findings of Section~\ref{sec:algebra}:

\begin{theorem} \label{thm:shift}
Consider the backward shift~$\sigma: \bounded \to \bounded$
and let~$\sigma^\infty (\bounded) = \bigcap_{k\geq 0} \sigma^k(\bounded)$.
Then for every~$0\leq k\leq \infty$ the set~$\sigma^k(\bounded)$ is an
$\ell^1$-dense subring of~$\bounded$ and the inclusions
\begin{equation} \label{eq:inclusions}
	\bounded \supset \sigma(\bounded) \supset \sigma^2(\bounded)
		\supset \cdots \supset \sigma^\infty (\bounded)
\end{equation}
are all strict. Moreover,
the elements of~$\sigma^k(\bounded)$ are the power series~$f(x)\in \bounded$ satisfying
\begin{equation} \label{eq:peano}
	f\p{\frac{1}{x}} = c_1 x + c_2 x^2 + \cdots + c_{k-1} x^{k-1} + \O (x^k) \as x\to 0
\end{equation}
for some~$c_1,\ldots,c_{k-1} \in \C$.
The restriction of~$\sigma$ to~$\sigma^\infty(\bounded)$ is one-to-one.
For every constant~$\lambda\neq 0$ the map~$\sigma-\lambda: \bounded\to\bounded$ is one-to-one.
\end{theorem}

Roughly speaking, the backward shift~$\sigma$ on~$\bounded$ is almost invertible:
it is one-to-one on a dense subset and for every small perturbation~$\lambda\neq 0$
the map~$\sigma-\lambda$ is one-to-one on the whole set.

By definition, the elements of~$\sigma^k (\bounded)$ are exactly the bounded power series
that can be extended $k$~times to the left to a bounded power series.
Such extension is unique for the exception of the constant term,
which does not affect boundedness but does affect further extensibility.
The strict inclusions~\eqref{eq:inclusions} state that for every~$k\geq 0$ there
is a bounded power series that can be extended $k$~times but not~$k+1$ times.
Since~$\sigma$ does not preserve products,
the fact that the set~$\sigma^k (\bounded)$ is a subring is not trivial.

Equation~\eqref{eq:peano} relates left-extensibility of a power series
to the behavior at~$x=\infty$ of the corresponding function.
If~$\O(x^k)$ is replaced by~$o(x^k)$, then~\eqref{eq:peano} turns into
the notion of Peano differentiability~\cite{fischer2005peano}.
In particular, with the exception of~$f=0$,
a power series~$f(x)\in \sigma^\infty(\bounded)$
is infinitely Peano differentiable (but not analytic) at~$x=\infty$.


\section{Boundedness and Coefficients} \label{sec:convexity}
In this section we are interested in asymptotic properties of the coefficient
sequence that are entailed by the boundedness of the resulting function.
We begin by reviewing some elementary facts about convergent power series.

\begin{lemma} \label{lem:max_inf}
Let~$f(x)=\sum_{n\geq 0} a_n x^n$ be a power series convergent on the real line and
not a polynomial. For every fixed~$x\in \R$, let~$\argmax_{n\geq 0} \abs{a_nx^n}$
denote the minimum index realizing the maximum~$\max_{n\geq 0}\abs{a_n x^n}$.
Then:
\begin{equation} \label{eq:max_argmax}
	\lim_{\abs{x} \to\infty} \max_{n\geq 0} \abs{a_n x^n} = \infty, \quad
	\lim_{\abs{x}\to\infty} \argmax_{n\geq 0} \abs{a_nx^n} = \infty.
\end{equation}
\end{lemma}
\begin{proof}
Convergence implies that for every fixed~$x\in\R$
the sequence~$(\abs{a_n x^n})_{n\geq 0}$ has a maximum.
Moreover, for every fixed~$n<m$ such that~$a_n,a_m\neq 0$ we have
\[
	\lim_{\abs{x}\to\infty} \abs{a_n x^n} = \infty, \quad
	\lim_{\abs{x}\to\infty} \abs{\frac{a_n x^n}{a_m x^m}} = 0.
\]
Since there are infinitely many
non-zero coefficients, the statement follows.
\end{proof}

Let~$f(x)=\sum_{n\geq 0} a_n x^n$ be
a power series with non-zero coefficients and convergent on the real line.
Then the sequence~$(\abs{a_n a_{n+1}^{-1}})_{n\geq 0}$ converges to~$+\infty$.
If~\eqref{eq:convexity} holds for some constants~$N\geq 1$ and~$0<\chi<1$,
then the sequence~$(\abs{a_n a_{n+1}^{-1}})_{n\geq 0}$ is strictly monotone
and, furthermore, for every~$m\geq N$
\begin{equation} \label{eq:interval_max}
	\abs{\frac{a_{m-1}}{a_{m}}} < x < \abs{\frac{a_{m}}{a_{m+1}}} \implies
		\abs{a_m x^m} = \max_{n\geq N} \abs{a_n x^n}.
\end{equation}
In~\eqref{eq:interval_max} the maximizer is unique; this represents the generic case.
If~$x=\abs{a_{m-1} a_{m}^{-1}}$ for some~$m\geq N$ there are two maximizers.
As a particular case of~\eqref{eq:interval_max},
in Lemma~\ref{lem:general_inequality} we will take~$x$
equal to the geometric mean of the endpoints.

\begin{lemma} \label{lem:general_inequality}
Let~$f(x)=\sum_{n\geq 0} a_n x^n$ be a power series convergent on the real line
and satisfying~\eqref{eq:convexity} for some constants~$N\geq 1$ and~$0<\chi<1$.
Fix~$m\geq N$ and let~$x=\abs{a_{m-1}a_{m+1}^{-1}}^{\frac{1}{2}}$. Then for every~$p\geq q\geq 0$
the following inequalities hold:
\begin{equation} \label{eq:general_inequality}
	\abs{\frac{a_{m+p} x^{m+p}}{a_{m+q} x^{m+q}}} \leq \chi^{\frac{(p-q)^2}{2}},
	\quad
	\abs{\frac{a_{m-p} x^{m-p}}{a_{m-q} x^{m-q}}} \leq \chi^{\frac{(p-q)^2}{2}}.
\end{equation}
\end{lemma}
\begin{proof}
Let us prove the first inequality in~\eqref{eq:general_inequality}.
Multiple applications~\eqref{eq:convexity} give
\[
	\abs{\frac{a_{m+p}}{a_{m+q}}}
	= \prod_{k=q}^{p-1} \abs{\frac{a_{m+k+1}}{a_{m+k}}}
	\leq \prod_{k=q}^{p-1} \p{ \chi^k \abs{\frac{a_{m+1}}{a_{m}}} }
	= \chi^{\frac{(p-q)(p+q-1)}{2}} \abs{\frac{a_{m+1}}{a_{m}}}^{p-q}.
\]
Multiplying both sides by~$x^{p-q} = \abs{a_{m-1}a_{m+1}^{-1}}^{\frac{(p-q)}{2}}$ gives
\[
	\abs{\frac{a_{m+p} x^{m+p}}{a_{m+q} x^{m+q}}}
	\leq
	\chi^{\frac{(p-q)(p+q-1)}{2}} \abs{\frac{a_{m+1} a_{m-1}}{a_{m}^2}}^{\frac{p-q}{2}}
	\leq \chi^{\frac{p^2-q^2}{2}} \leq \chi^{\frac{(p-q)^2}{2}}.
\]
The second inequality in~\eqref{eq:general_inequality}
can be obtained in a similar way.
\end{proof}

\begin{remark} \label{rem:Legendre}
The result of Lemma~\ref{lem:general_inequality} can be interpreted as
a discrete Taylor-like inequality.
Letting~$g(n)=-\log(\abs{a_n})$ and taking~$q=0$ in~\eqref{eq:general_inequality} gives
\[
	g(m+p) \leq g(m) + \p{\frac{g(m+1)-g(m-1)}{2}} p + \log(\chi^{-1}) \frac{p^2}{2}.
\]
Define the discrete derivative of a sequence~$g$ as~$(\Delta g)(n) = g(n+1)-g(n)$.
Then the quantity~$\frac{1}{2}(g(m+1)-g(m-1))$ is the average of~$\Delta g(m)$
and~$\Delta g(m-1)$, while~$\log(\chi^{-1})$
is bounded by~$\Delta^2 g(n-1)$. Let~$x=\exp(s)$.
Then~$\abs{a_n x^n} = \exp(ns-g(n))$ and the maximum in~\eqref{eq:interval_max}
can be interpreted as a discrete Legendre transform.
\end{remark}

\begin{proof} [Proof of Theorem~\ref{thm:convexity}]
Let~$f(x)=\sum_{n\geq 0} a_n x^n$ be a real power series,
convergent on the real line, and satisfying~\eqref{eq:convexity} for
some constants~$N\geq 1$ and~$0<\chi< 1$.

We will prove that if~$\lim_{n\to\infty} L(n) \leq \vartheta(\chi)$
or if~$\lim_{n\to\infty} L(n) > 2\psi(\chi)$,
then the function~$f:\R\to\R$ is unbounded on the real line. We will do so by
exhibiting a sequence~$(x_m)_{m\geq N}$ and a constant~$C>0$
such that~$\lim_{m\to\infty} x_m=+\infty$ and
\begin{equation} \label{eq:convexity_goal}
	\abs{f(x_m)} \geq C \max_{n\geq N} \abs{a_n x_m^n}
\end{equation}
for every~$m$ sufficiently large. By Lemma~\ref{lem:max_inf}, this implies that~$f$
is unbounded. Note that for every fixed~$M\geq 0$ we have
\begin{equation} \label{eq:convexity_reduction}
	\lim_{x\to +\infty}
		\frac{\sum_{0\leq k\leq M} \abs{a_k x^k}}{\max_{n\geq 0} \abs{a_n x^n}}
		= 0.
\end{equation}
Therefore, in obtaining~\eqref{eq:convexity_goal} we can always discard a finite number of terms.
In particular, without loss of generality, we can assume~$N=0$ in~\eqref{eq:convexity}.

Define
\[
	\Theta = \sum_{1\leq k\leq \vartheta(\chi)} \chi^{\frac{k^2}{2}},
	\quad
	\Psi = \sum_{k>\psi(\chi)} \chi^{\frac{k^2}{2}}.
\]
By definition of~$\vartheta(\chi)$ and~$\psi(\chi)$ it follows that~$\Theta<1$
and~$\Psi<\frac{1}{2}$ respectively.

Suppose that~$\lim_{n\to\infty} L(n) \leq \vartheta(\chi)$,
that is, $L(n)\leq \vartheta(\chi)$ for every~$n$ sufficiently large.
In light of~\eqref{eq:convexity_reduction}, without loss of generality
we can assume~$L(n)\leq \vartheta(\chi)$ for every~$n\geq 0$, that is,
every subsequence of consecutive coefficients with the same sign
has length at most~$\vartheta(\chi)$.
If coefficients are eventually alternating signs, that is, $a_n a_{n-1}<0$ for every~$n$ sufficiently
large, then the coefficients of~$f(-x)$ have eventually all the same sign, thus~$f(-x)$
is unbounded, and therefore~$f(x)$ is unbounded.
Otherwise, there are infinitely many~$m\geq 0$ such that~$a_m$ and~$a_{m+1}$ have the same sign and,
without loss of generality, we assume them positive.
Fix such~$m$ and let~$x_m = (a_{m-1}a_{m+1}^{-1})^{\frac{1}{2}}$.
Let~$p,n\geq 1$ be any integers
such that~$a_{m+p}>0$ and~$a_{m+p+1},\ldots,a_{m+p+n}<0$.
By hypothesis~$n\leq \vartheta(\chi)$. From the first inequality of~\eqref{eq:general_inequality}
we obtain
\[
	\abs{\sum_{k=1}^{n} a_{m+p+k} x_{m}^{m+p+k}}
	\leq \p{\sum_{1\leq k\leq n} \chi^{\frac{k^2}{2}}}
	\abs{a_{m+p} x_m^{m+p}}
	\leq \Theta \abs{a_{m+p} x_m^{m+p}}.
\]
It follows that every sequence of consecutive negative terms appearing after the index~$m+1$
is dominated by the positive term preceding it. Therefore
\[
	\sum_{n\geq m+1} a_n x_m^n > \p{1-\Theta} a_{m+1} x_m^{m+1} > 0.
\]
Similarly, from the second inequality of~\eqref{eq:general_inequality} we
obtain~$\sum_{1\leq n\leq m} a_n x_m^n > \p{1-\Theta} a_{m} x_m^{m}$.
We conclude that
\[
	f(x_m) = \sum_{n\geq 0} a_n x^n >
	\p{1-\Theta} a_{m} x_m^{m} = \p{1-\Theta} \max_{n\geq 0} \abs{a_{n} x_m^{n}}.
\]
Since~$\lim_{m\to\infty} x_m = +\infty$
and~$\Theta<1$,
by Lemma~\ref{lem:max_inf} it follows that~$f$ is unbounded.

Now suppose that~$\lim_{n\to\infty} L(n) > 2\vartheta(\chi)$,
that is, $L(n)\geq 2\vartheta(\chi)+1$ eventually.
Then there are infinitely many~$m\geq 0$ such
that the coefficients~$a_{m-\vartheta(\chi)},\ldots,a_{m+\vartheta(\chi)}$ have the
same sign and, without loss of generality, we can suppose them positive.
Fix such~$m$ and let~$x_m = (a_{m-1}a_{m+1}^{-1})^{\frac{1}{2}}$.
From the first inequality of~\eqref{eq:general_inequality} we obtain
\[
	\abs{\sum_{k>m+\vartheta(\chi)} a_{k} x_m^{k}}
		\leq
		\abs{\sum_{k>\vartheta(\chi)} \chi^{\frac{k^2}{2}}} \abs{a_{m} x_m^{m}}
		= \Psi \abs{a_{m} x_m^{m}}.
\]
Similarly, from the second inequality of~\eqref{eq:general_inequality}
we obtain~$\abs{\sum_{1\leq k < m-\vartheta(\chi)} a_{k} x_m^{k}} \leq \Psi \abs{a_{m} x_m^{m}}$.
We conclude that
\[
	f(x_m) \geq (1-2\Psi) \abs{a_{m} x_m^{m}} = (1-2\Psi) \max_{n\geq 0} \abs{a_{n} x_m^{n}}.
\]
Since~$\lim_{m\to\infty} x_m = +\infty$
and~$\Psi<\frac{1}{2}$,
by Lemma~\ref{lem:max_inf} it follows that~$f$ is unbounded.
\end{proof}

We conclude this session with some open problems.
As discussed in the introduction, bounded power series on the real line cannot have
arbitrary convexity.

\begin{definition}
Let~$\tilde \chi$ denote the supremum of the constants~$\chi$ such that
if a complex power series satisfies~\eqref{eq:convexity} for some~$N\geq 1$, then it is unbounded
on the real line.
Let~$\tilde \chi_\R$ denote the supremum of the constants~$\chi$ such that
if a real power series satisfies~\eqref{eq:convexity} for some~$N\geq 1$, then it is unbounded
on the real line.
\end{definition}

Clearly~$\tilde \chi \leq \tilde \chi_\R$.
The following example shows that~$\tilde\chi_\R \leq 1$, and thus~$\tilde\chi\leq 1$.

\begin{example}
The coefficients of the power series associated to~$\sin(x)+\cos(x)$
satisfy~$\abs{a_n}=\frac{1}{n!}$, thus~\eqref{eq:convexity}
with~$N=1$ and~$\chi=1$.
\end{example}

Corollary~\ref{cor:convexity} in the introduction
gives~$\tilde\chi_\R\geq \frac{1}{2}$.
Its proof relies on
the inequalities~$\vartheta(\chi)\geq 2$ and~$\psi(\chi)\leq 1$, which can be
seen numerically to hold for~$\chi \leq 0.67522$.
The argument of the second part of the proof
of Theorem~\ref{thm:convexity} generalizes without modifications to the case of complex
coefficients, showing that if~$\psi(\chi)=0$ then~$f$ is unbounded on the real line.
We see that~$\psi(\chi)=0$ holds for~$\chi \leq 0.207875$.
We are unable to close the gaps~$0.67522 \leq \tilde \chi_\R \leq 1$
and~$0.207875 \leq \tilde\chi\leq 1$.

\begin{conjecture} \label{conj:equal}
The constants~$\tilde \chi$ and~$\tilde \chi_\R$ are equal to each other.
\end{conjecture}

\begin{conjecture} \label{conj:one}
The constants~$\tilde \chi$ and~$\tilde \chi_\R$ are equal to~$1$.
\end{conjecture}

The techniques developed in this section might help deducing asymptotic inequalities
of combinatorial sequences from the boundedness of their generating function.
For example, let~$\exp(1-\exp(x))=\sum_{n\geq 0} \frac{b_n}{n!} x^n$.
The quantity~$b_n$, known as
\emph{complementary Bell number} or
\emph{Uppuri-Carpenter number}~\cite{uppuluri1969numbers},
is equal to the difference between the number of partitions of~$\{1,\ldots,n\}$
with an even number of blocks and those with an odd number of blocks.
We make the following conjecture:

\begin{conjecture} \label{conj:combinatorics}
The coefficient sequence of the power series~$\exp(1-\exp(x))$
contains arbitrarily long subsequences of consecutive coefficients with the same sign.
\end{conjecture}

\begin{remark} \label{rem:combinatorics}
Wilf's Conjecture states that~$b_n\neq 0$ for~$n>2$~\cite{de2016wilf, amdeberhan2013complementary}.
Since the function~$\exp(1-\exp(x))$ is bounded on the real line,
if both Wilf's Conjecture and Conjecture~\ref{conj:combinatorics} are true, then
by Corollary~\ref{cor:signs} we have
\[
	\limsup_{n\to\infty} \abs{\frac{n b_{n+1} b_{n-1}}{(n+1)b_n^2}} \geq 1.
\]
In contrast, note that the
inequality~$\liminf_{n\to\infty} \abs{\frac{n b_{n+1} b_{n-1}}{(n+1)b_n^2}} \leq 1$
follows from the elementary fact that the power series is convergent on the real line.
\end{remark}


\section{The Topology of Bounded Power Series} \label{sec:topology}

Let~$\bounded$ denote the set of complex power series bounded on the real line.
This section investigates topological aspects of~$\bounded$.
Note that if a power series~$f(x)$ converges for every~$x \in \R$, then
it converges for every~$z\in \C$. Therefore, there is a natural bijection
between the set~$\bounded$ and the set of entire functions
that bounded on the real line. On the other hand, there is a natural bijection
between the set~$\bounded$ and the set of complex sequences~$(a_n)_{n\geq 0}$ satisfying
\[
	a_n = - \sum_{k\geq 1} a_{n+k} x^k + \O\p{\frac{1}{x}}, \as \abs{x}\to\infty
\]
for every~$n\geq 1$ or, equivalently, for~$n=1$.
The set~$\bounded$ has the cardinality of the continuum: to see this, note
that~$\abs{\C}\leq \abs{\bounded} \leq \abs{\C^\N}$ and~$\abs{\C}=\abs{\C^\N}=\abs{\R}$.

There are three natural notions of convergence on~$\bounded$: first,
$\ell^1$-convergence of coefficient sequences; second,
uniform convergence of bounded functions on~$\R$; third,
uniform convergence of entire functions on compact subsets of~$\C$.
Theorem~\ref{thm:topology} states that these topologies are inequivalent and incomplete,
and describes their topological completions.

\begin{proof} [Proof of Theorem~\ref{thm:topology}]
First, let us prove that the set~$\bounded$ endowed with~$\norm{\cdot}_{\ell^1}$
is (isometric to) a dense subset of~$\ell^1(\N,\C)$.
Let~$f\in \bounded$. Note that~$\norm{f}_{\ell^1}$ is well-defined: since~$f(x)$
converges for every~$x\in \R$, then it converges absolutely for~$x=1$.
Identify power series with their coefficient sequences.
Since the sequences with exactly one non-zero coefficient form a dense subset of~$\ell^1(\N,\C)$,
it is enough to prove that every monomial~$z^m$ can be $\ell^1$-approximated
by an entire function bounded on the real line.
Fix~$\lambda>0$ and consider~$z^m \exp(-\lambda z^{2})$.
Since product of functions corresponds to convolution
of coefficient sequences, Young's convolution inequality~\cite{quek1983sharpness} gives
\[
	\norm{z^m - z^m \exp(-\lambda z^2)}_{\ell^1} \leq
	\norm{z^m}_{\ell^1} \norm{1-\exp(-\lambda z^{2})}_{\ell^1}
	= \exp(\lambda) - 1.
\]
As~$\lambda \to 0$ the function~$z^m \exp(-\lambda z^{2})$ converges to~$z^m$
in the $\ell^1$-norm.
Since the function~$z^m \exp(-\lambda z^2)$ is bounded on the real line,
it follows that~$\bounded$ is dense in~$\ell^1(\N,\C)$.

Density of~$\bounded$ in~$C^\mathrm{b}(\R,\C)$ with respect to the uniform norm on~$\R$
is an immediate consequence of
Carleman's Approximation Theorem~\cite{carleman1928theoreme, kaplan1955approximation},
which states that every continuous function~$\R\to\C$ can be uniformly approximated
on~$\R$ by an entire function.

Let us now prove that~$\bounded$ endowed with uniform convergence on compact subsets of~$\C$
is dense in the set of entire functions.
Since entire functions can be uniformly approximated
by polynomial on compact sets, it is enough to show that every monomial~$z^m$
can be uniformly approximated on compact sets by an entire function which is bounded on the real
line. Fix~$r>0$ and note that for every~$\lambda>0$
\[
	\sup_{\abs{z}\leq r} \abs{z^m - z^m \exp(-\lambda z^2)} \leq r^m \p{\exp(\lambda r^2)-1}.
\]
The right hand side converges to~$0$ as~$\lambda \to 0$. This proves that~$\bounded$
is dense in the set of entire functions.

It remains to show that the three topologies are inequivalent.
For every~$m\geq 1$ the functions~$0$, $1$, $\exp(-\frac{z^2}{m})$, $\frac{1}{m}\sin(mz)$,
and~$\exp(-z^{2m})$ are entire and bounded on the real line.
As~$m\to\infty$ the following facts hold: the function~$\exp(-\frac{z^2}{m})$
converges uniformly to~$1$ on compact subsets of~$\C$ but does not converge uniformly on~$\R$;
the function~$\frac{1}{m}\sin(mx)$ converges uniformly to~$0$ on~$\R$ while
$
	\norm{\frac{1}{m}\sin(mx)}_{\ell^1} \geq 1
$
for every~$m$;
the function~$\exp(-z^{2m})$ converges to~$1$ in~$\ell^1$, but
does not converge uniformly on compact subsets of~$\C$ since, in particular,
it does not converge pointwise at~$z=i$.
\end{proof}


\section{The Algebra of Bounded Power Series} \label{sec:algebra}
Sum and product of functions make the set~$\bounded$ a ring.
Since entire functions form an integral domain, the ring~$\bounded$ is an integral domain.
As shown in the introduction, the ring~$\bounded$ is invariant under
backward shift~$\sigma\p{(a_n)_{n\geq 0}} = (a_{n+1})_{n\geq 0}$.
Properties of the ring~$\bounded$ and the backward shift are collected in
Theorem~\ref{thm:shift} in the introduction and presented
in this section as a series of individual results.

While~$\sigma^k$ acts on~$\bounded$ by forgetting the first $k$~elements,
all but one (namely~$a_0$) can be recovered from the tail.
This fact translates into algebraic language as follows:

\begin{lemma} \label{lem:ker_sigma}
For every~$k\geq 1$ we have~$\ker \sigma^k = \ker \sigma$.
\end{lemma}
\begin{proof}
If~$f(x)=\sum_{n\geq 0} a_n x^n \in \ker \sigma^k$, then
the sequence~$(a_n)_n$ is eventually~$0$, that is, the function~$f$ is a bounded polynomial,
and thus a constant.
It follows that~$f(x) \in \ker \sigma$. This proves~$\ker \sigma^k \subseteq \ker \sigma$.
The converse~$\ker \sigma \subseteq \ker \sigma^k$ is trivial.
\end{proof}

The map~$\sigma:\bounded\to\bounded$ is not surjective,
but it is not far from being surjective:

\begin{proposition} \label{prop:strict_inclusions}
Define~$\sigma^\infty (\bounded) = \bigcap_{k\geq 0} \sigma^k(\bounded)$.
The inclusions 
\begin{equation}
	\bounded \supset \sigma(\bounded) \supset \sigma^2(\bounded)
		\supset \cdots \supset \sigma^\infty (\bounded)
\end{equation}
are all strict.
For every~$0\leq k\leq \infty$ the set~$\sigma^k(\bounded)$ is $\ell^1$-dense in~$\bounded$.
\end{proposition}
\begin{proof}
Let~$k\geq 0$. Identify the function~$\sin(x)$ with its power series.
We claim that~$\sigma^k\sin(x) \in  \sigma^k(\bounded)\setminus \sigma^{k+1}(\bounded)$.
Clearly we have~$\sigma^k\sin(x) \in  \sigma^k(\bounded)$.
For the sake of contradiction, assume that~$\sigma^k\sin(x) \in \sigma^{k+1}(\bounded)$.
Then there is~$g(x)\in \bounded$
such that~$\sigma^k \sin(x) = \sigma^k (\sigma g(x))$ and,
by Lemma~\ref{lem:ker_sigma}, there is~$a\in \C$ such that
\[
	\sin(x) - a = \frac{g(x)-g(0)}{x}.
\]
This is impossible since the right hand side converges to~$0$ as~$x\to\infty$.

It remains to prove that~$\sigma^\infty (\bounded)$ is $\ell^1$-dense in~$\bounded$.
Let~$m\in \N$ and~$\lambda>0$.
Since the function~$x^k (x^m \exp(-\lambda x^2))$ is bounded for every~$k\geq 0$,
we have~$x^m \exp(-\lambda x^2) \in \sigma^\infty (\bounded)$.
By the proof of Theorem~\ref{thm:topology} the
set~$\{z^m \exp(-\lambda z^2)\}_{\lambda>0,m\in \N}$
is dense in~$\ell^1(\N,\C)$, thus in~$\bounded$. This shows that~$\sigma^\infty (\bounded)$
is dense in~$\bounded$. It follows that every~$\sigma^k(\bounded)$ is dense in~$\bounded$.
\end{proof}

The following result characterizes the elements of~$\sigma^k(\bounded)$:

\begin{proposition} \label{prop:peano}
Let~$f\in \bounded$ and~$k\geq 1$. Then~$f\in \sigma^k(\bounded)$ if and only if
there are~$c_1,\ldots,c_{k-1} \in \C$ such that
\begin{equation} \label{eq:peano_prop}
	f\p{\frac{1}{x}} = c_1 x + c_2 x^2 + \cdots + c_{k-1} x^{k-1} + \mathcal O (x^k) \as x\to 0.
\end{equation}
In particular, the set~$\sigma^k(\bounded)$ is a subring of~$\bounded$.
\end{proposition}
\begin{proof}
We have~$f(x)\in \sigma^k (\bounded)$
if and only if there are~$a_0,\ldots,a_{k-1} \in \C$ such that
\[
	a_0 + a_1 x + \cdots a_{k-1} x^{k-1} + x^k f(x) = \mathcal O(1) \as x\to 0,
\]
which is the same as~\eqref{eq:peano_prop} with~$c_j=-a_{k-j}$. It remains to prove
that~$\sigma^k(\bounded)$ is a ring.
Since~$\sigma$ is linear, the set~$\sigma^k(\bounded)$ is closed under sums.
From~\eqref{eq:peano_prop} it follows that
\[
	\sigma^k (\bounded) \sigma^h (\bounded) \subseteq \sigma^{\min(k+1,h+1)} (\bounded).
\]
for every~$k,h\geq 1$.
Taking~$k=h$ shows that~$\sigma^k(\bounded)$ is closed under products.
\end{proof}

We will need the following algebraic result, which we state in general:

\begin{lemma} \label{lem:B_algebra}
Let~$(B,+)$ an abelian group and~$\sigma$ an endomorphism
satisfying~$\ker \sigma^k = \ker \sigma$ for every~$k\geq 1$.
Define~$\sigma^\infty (B)= \bigcap_{k\geq 0} \sigma^k(B)$.
Then the restriction of~$\sigma$ to~$\sigma^\infty (B)$ is one-to-one.
\end{lemma}
\begin{proof}
From~$\ker \sigma = \ker \sigma^2$ it follows that
\begin{equation} \label{eq:ker_cap_im}
	\ker\sigma \cap \sigma(B) = \{0\},
\end{equation}
which in turn implies~$\ker \sigma \cap \sigma^\infty (B) = \{0\}$,
proving that the restriction of~$\sigma$ to~$\sigma^\infty (B)$ is injective.
It remains to prove surjectivity. Let~$a \in \sigma^\infty (B)$.
Since~$a\in \sigma^2(B)$, there is~$b\in \sigma(B)$ such that~$\sigma(b)=a$.
We claim that~$b\in \sigma^\infty (B)$.
Fix an arbitrary~$k\geq 1$.
Since~$a\in \sigma^{k+1}(B)$, there is~$d\in B$ such that~$a=\sigma^{k+1}(d)$.
It follows that~$b-\sigma^k(d) \in \ker\sigma \cap \sigma(B)$ and therefore,
from~\eqref{eq:ker_cap_im}, that~$b=\sigma^k(d)$.
Since~$k$ was arbitrary, it follows that~$b\in \sigma^\infty (B)$.
This proves surjectivity.
\end{proof}

From Lemma~\ref{lem:ker_sigma} and Lemma~\ref{lem:B_algebra} it immediately follows that:

\begin{corollary} \label{cor:one-to-one}
The restriction of~$\sigma$ to~$\sigma^\infty(\bounded)$ is one-to-one.
\end{corollary}

It is well known that the point spectrum of backward shift on~$\ell^1(\N,\C)$
is equal to the interior of the unit circle.
As we will see,
restricting the backward shift from~$\ell^1(\N,\C)$ to its dense subset~$\bounded$
restricts the point spectrum to a single eigenvalue.

\begin{proposition} \label{prop:spectrum}
The only eigenvalue of~$\sigma: \bounded \to \bounded$ is~$0$.
For every~$\lambda\neq 0$ the map~$\sigma-\lambda: \bounded \to \bounded$
is one-to-one.
\end{proposition}
\begin{proof}
Let~$\lambda\in \C\setminus \{0\}$. Let~$f\in \bounded$
be a solution of~$\sigma f = \lambda f$. Then
\[
	f(z) = \frac{f(0)}{1-\lambda z}.
\]
If~$f(0)\neq 0$, this function has a pole at~$z=\lambda^{-1}$ and its power series cannot converge
on the whole real line. Therefore~$f=0$.
This proves that~$\lambda$ is not an eigenvalue.

It remains to prove that~$\sigma-\lambda: \bounded \to \bounded$ is surjective.
For every~$g\in \bounded$ the equation~$(\sigma - \lambda) f = g$ has solution
\[
	f(z) = \frac{-\lambda^{-1} g(\lambda^{-1}) + z g(z)}{1-\lambda z}.
\]
Let us prove that~$f\in\bounded$. Since the numerator has a zero of order at least~$1$
at~$z=\lambda^{-1}$, the function~$f$ is entire.
Since~$g$ is bounded on the real line, so is $f$.
\end{proof}

\begin{proof} [Proof of Theorem~\ref{thm:shift}]
Theorem~\ref{thm:shift} is the collection of Proposition~\ref{prop:strict_inclusions},
Proposition~\ref{prop:peano}, Corollary~\ref{cor:one-to-one},
and Proposition~\ref{prop:spectrum}.
\end{proof}


\bibliographystyle{plain}
\bibliography{refs}

\end{document}